\documentclass[11pt,reqno]{amsart}
\usepackage{amsmath}
\usepackage{amssymb}
\usepackage{amsthm}
\usepackage{enumerate}
\usepackage[usenames]{color}
\usepackage{eepic,epic}

\usepackage{epstopdf}

\usepackage{graphicx}

\newtheorem{thm}{Theorem}[section]

\newtheorem{lem}[thm]{Lemma}
\newtheorem{prop}[thm]{Proposition}

\theoremstyle{definition}

\newtheorem{step}{Step}
\theoremstyle{remark}
\newtheorem{rem}[thm]{Remark}
\numberwithin{equation}{section}

\newcommand{\R}{\mathbb{R}}

\begin{document}
\title[]
{Convergence rate of the finite element approximation for extremizers of Sobolev inequalities} 
\author{Woocheol Choi}
\address{Department of Mathematics Education, Incheon National University, Incheon 22012, Korea}
\email{choiwc@inu.ac.kr}

\author{Younghun Hong}
\address{Department of Mathematics, Chung-Ang University, Seoul 06974, Korea}
\email{yhhong@cau.ac.kr}

\author{Jinmyoung Seok}
\address{Department of Mathematics, Kyonggi University, Suwon 16227, Korea}
\email{jmseok@kgu.ac.kr}



\maketitle
\begin{abstract}
In this paper, we are concerned with the convergence rate of a FEM based numerical scheme approximating extremal functions of the Sobolev inequality. 
We prove that when the domain is polygonal and convex in $\R^2$,
the convergence of a finite element solution to an exact extremal function in $L^2$ and $H^1$ norms has the rates $O(h^2)$ and $O(h)$ respectively, where $h$ denotes the mesh size of a triangulation of the domain. 
\end{abstract}

\section{Introduction}
Let $\Omega \subset \R^N$ be a bounded domain, where $N \geq 2$. In this paper, we are concerned with the Sobolev inequality
\[
C(\Omega, p)\|u\|_{L^p(\Omega)} \leq \|\nabla u\|_{L^2(\Omega)},
\]
where $p \in (2,\, 2N/(N-2))$ for $N \geq 3$ and $p \in (2,\, \infty)$ for $N = 1, 2$.
It is well known that the best constant $C(\Omega, p)$, which is given by the infimum of the following minimization problem
\begin{equation}\label{eq-0-2}
C(\Omega, p) = \inf\left\{\frac{\|\nabla u\|_{L^2(\Omega)}}{\|u\|_{L^p(\Omega)}} ~\Big|~ u \in H^1_0(\Omega), u \neq 0\right\},
\end{equation}
is attained by a positive function $U_{\Omega, p}$ satisfying the semi-linear elliptic equation
\begin{equation}\label{Lane-Emden}
-\Delta u = |u|^{p-2}u \ \text{ in } \Omega, \quad u  \in H^1_0(\Omega).
\end{equation}
The aim of this paper is to obtain a sharp convergence rate of a numerical scheme for approximating the minimizer $U_{\Omega, p}$. This work is motivated by Tanaka-Sekine-Mizuguchi-Oishi \cite{TSMO} where they established convergence estimate for the best constant of the sobolev embedding $H_0^1 (\Omega) \rightarrow L^{p}(\Omega)$.


Now, we fix a polygonal convex domain $\Omega \subset \R^2$ and arbitrary $p \in (2, \infty)$. 
Let $\{T_h\}$ with $h>0$ be a family of regular triangulations of $\Omega$. (For the definition, we refer to \cite{Bartels}.)  
The finite element space $V_h \subset H_0^1 (\Omega)$ is given by
\begin{equation*}
V_h = \left\{ v \in H_0^1(\Omega) ~|~ \textrm{$v$ is a polynomial of degree $\leq 1$ on each $T \in T_h$}\right\}.
\end{equation*}
Define the following minimization problem on $V_h$, 
\begin{equation}\label{eq-0-1}
C_h(\Omega, p) = \min\left\{ \frac{\|\nabla \phi_h\|_{L^2} }{\|\phi_h\|_{L^p}} ~\Big|~ \phi_h \in V_h, \phi_h \neq 0 \right\}.
\end{equation}
Since $V_h$ is finite dimensional, it is complete with respect to $H^1_0$ norm.  
Then a standard argument showing the existence of a minimizer of \eqref{eq-0-2} applies in same manner to show the existence of a minimizer $U_h$ of the problem \eqref{eq-0-1}.

By the Lagrange multiplier theorem, it is easy to see that there exists a constant $\lambda_h >0$ such that
\begin{equation}\label{Euler-Lagrange-fem}
\int_{\Omega} \nabla U_h \nabla \phi_h dx = \lambda_h \int_{\Omega} |U_h|^{p-2} U_h \, \phi_h dx \quad \forall ~\phi_h \in V_h.
\end{equation}
Note that we may assume $\lambda_h = 1$ by redefining $U_h$ by $(\|\nabla U_h\|_{L^2}^2/\|U_h\|_{L^p}^p)^{\frac{1}{p-2}}U_h$.


\begin{thm}\label{thm-1} Assume that $\Omega \subset \R^2$ is a bounded convex domain with a polygonal boundary and $p > 2$.
Let $\{U_h\}$ be a family of minimizers of the problem \eqref{eq-0-1} with $\lambda_n = 1$ in \eqref{Euler-Lagrange-fem}
and $U_0 \in H^1_0(\Omega)$ be a unique positive minimizer of the problem \eqref{eq-0-2} satisfying \eqref{Lane-Emden}.
Then the following statements hold true:
\begin{enumerate}[{\em (i)}]
\item For any sequence $\{h_n\} \to 0$, 
$\{U_{h_n}\}$ converges to either $U_0$ or $-U_0$ in $H^1_0(\Omega)$ by choosing a subsequence. 
\item There exists a universal constant $C > 0$ such that for any sequences $\{h_n\} \to 0$ and $\{U_{h_n}\} \to U_0$, there holds
\begin{equation}\label{eq-16}
\|U_{h_n}-U_0\|_{L^2} \leq  Ch_n^2 \quad \textrm{and}\quad \|U_{h_n}-U_0\|_{H^1} \leq Ch_n.
\end{equation}
\item The $L^\infty$ norm of $U_h$ is uniformly  bounded, i.e., there exists a universal constant $C > 0$ such that
\[
\|U_h\|_{L^\infty(\Omega)} \leq C.
\]
\end{enumerate}
\end{thm} Also, it is worth to mention that there has been research to develop numerical scheme to find solutions to the nonlinear problem \eqref{Lane-Emden} (see \cite{CM, LZ, FJ} and references therein). The scheme based on mountain pass principle was developed by Choi-McKenna \cite{CM} to find a minimizer and it was extended by Li-Zhou \cite{LZ} to find multiple solutions. In \cite{FJ}, Faou and J\'ezquel  proved the exponential convergence rate for the normalized gradient algorithm for the nonlinear Schr\"odinger equation. Up to the author's best knowledge, there is no result on the convergence estimate between the solution to \eqref{Lane-Emden} and the finite element solution of the  discrete problem \eqref{Euler-Lagrange-fem}. Theorem \ref{thm-1} gives the corresponding estimate for two dimsional convex polygon. The key part of the proof of Theorem \ref{thm-1} is to use the non-degenaracy property of the minimizer. For this part, we modified some ideas in our previous work \cite{CHS} where we studied the convergence estimate for the nonrelativistic limit of the nonlinear pseudo-relativisitic equations.

The rest of the paper is organized as follows. Section 2 is devoted to prove $H^1$ convergence of a approximate solution $U_h$. In Sections 3 and 4, we shall obtain the convergence rates of $U_h$ in $H^1$ and $L^2$ respectively. In Section 5, we prove the uniform $L^\infty$ boundedness of $U_h$.   
It is shown in Section 6 that there is a  good agreement between our analytic results and the real numerical implementation.
The finial section is an appendix which collects useful analytic tools frequently invoked in preceding sections.

\section{Convergence of $U_h$ in $H^1_0$ space}
In this section, we prove the $H^1$ convergence of $U_h$ through several steps.
We recall that
\[
C(\Omega,p) = \min_{v \in H_0^1 (\Omega)\setminus\{0\}} \frac{\|v\|_{H_0^1 (\Omega)}}{\|v\|_{L^p(\Omega)}} 
\quad \textrm{and}\quad 
C_h(\Omega,p) =\min_{v \in V_h \setminus\{0\}}\frac{\|v\|_{H_0^1 (\Omega)}}{\|v\|_{L^p(\Omega)}},
\]
where we imposed the norm $\|\nabla\cdot\|_{L^2(\Omega)}$ on $H^1_0(\Omega)$.
We simply denote $C(\Omega,p)$ and $C_h(\Omega,p)$ by $C_0$ and $C_h$ respectively.

\begin{step}
The value $C_h$ converges to $C_0$ as $h \to 0$. 
\end{step}
\begin{proof}
Since $V_h \subset H^1_0(\Omega)$, one has $C_0 \leq C_h$.
From Proposition \ref{approx} and Proposition \ref{regularity}, we can choose some $\psi_h \in V_h$ satisfying
$\|U_0-\psi_h\|_{H_0^1 (\Omega)} \leq Ch$ for some $C > 0$ independent of $h$.
Then we see that for small $h > 0$,
\begin{equation*}
\begin{aligned}
C_0 &= \frac{\|U_0\|_{H_0^1 (\Omega)}}{\|U_0\|_{L^{p} (\Omega)}} 
\geq \frac{\|\psi_h\|_{H_0^1 (\Omega)} - Ch}{\|\psi_h\|_{L^{p}(\Omega)}+ Ch} \\
&\geq \frac{\|\psi_h\|_{H_0^1 (\Omega)}}{\|\psi_h\|_{L^{p}(\Omega)}} 
- C\frac{\|\psi_h\|_{H_0^1 (\Omega)}+\|\psi_h\|_{L^2 (\Omega)}}{\|\psi_h\|_{L^2(\Omega)}^2}h \\
&\geq C_h + O(h),
\end{aligned}
\end{equation*}
which shows that $\lim_{h\to0}C_h = C_0$.
\end{proof}

\begin{step}
For any sequence $\{h_n\} \to 0$, $\{U_{h_n}\}$ converges in $H^1_0(\Omega)$ to some nonzero function $W_0 \in H^1_0(\Omega)$ after choosing a subsequence. 
\end{step}
\begin{proof}
By the above Step 1, note that for small $h > 0$,
\begin{equation}\label{eq-31}
C_0 \leq \frac{\|U_h\|_{H_0^1 (\Omega)}}{\|U_h\|_{L^{p} (\Omega)}} \leq C_0+1.
\end{equation}
By setting $\phi_h = U_h$ in \eqref{Euler-Lagrange-fem}, we get
\begin{equation}\label{lem-eq-0}
\int_{\Omega} |\nabla U_h|^2 dx = \int_{\Omega} U_h^{p} dx.
\end{equation}
Combining this with \eqref{eq-31}, we obtain that for small $h > 0$,
\begin{equation}\label{lem-eq-1}
C_0 < \|U_h\|_{L^p(\Omega)}^{\frac{p}{2}-1}, \quad  \|U_h\|_{H^1_0(\Omega)}^{1-\frac{2}{p}} < C_0+1.
\end{equation}
The second inequality of \eqref{lem-eq-1} and the compactness of the embedding $H^1_0\hookrightarrow L^p$
says that for any $\{h_n\} \to 0$, $\{U_h\}$ converges to some $W_0$ weakly in $H^1_0$ and strongly in $L^p$ after choosing a subsequence. 
From the first inequality in \eqref{lem-eq-1}, we then deduce that $W_0$ is nonzero. 
Moreover, we see from Proposition \ref{approx} that there exists a sequence $\psi_{h_n} \in V_{h_n}$ such that 
$\|W_0-\psi_{h_n}\|_{H^1_0} = o(1)$ so one has
\begin{equation}\label{lem-eq-2}
\begin{aligned}
\|\nabla W_0\|_{L^2}^2 &= \lim_{n\to\infty}\int_{\Omega}\nabla U_{h_n}\cdot\nabla W_0\,dx \\
&=\lim_{n\to\infty}(\int_{\Omega}\nabla U_{h_n}\cdot\nabla \psi_{h_n}\,dx +\int_{\Omega}\nabla U_{h_n}\cdot\nabla(W_0-\psi_{h_n})\,dx) \\
&=\lim_{n\to\infty}(\int_{\Omega}|U_{h_n}|^{p-2}U_{h_n}\psi_{h_n}\,dx+o(1)) \\
&=\lim_{n\to\infty}(\int_{\Omega}|U_{h_n}|^{p-2}U_{h_n}W_0\,dx+o(1))
= \|W_0\|_{L^p}^p.
\end{aligned}
\end{equation}
Then, the equality \eqref{lem-eq-0} implies that
\[
\|\nabla U_{h_n}\|_{L^2}^2 = \lim_{n\to\infty}\|U_{h_n}\|_{L^p}^p  = \|W_0\|_{L^p}^p = \|\nabla W_0\|_{L^2}^2.
\]
From this and the fact that $\{U_{h_n}\}$ converges weakly to $W_0$, we conclude that the sequence $\{U_{h_n}\}$ strongly converges to $W_0$ in $H_0^1 (\Omega)$. 
\end{proof}

\begin{step} 
The function $W_0$ is either $U_0$ or $-U_0$.
\end{step}
\begin{proof}
Fix an arbitrary $\psi \in H^1_0(\Omega)$. Then by choosing $\psi_{h_n} \in V_{h_n}$ satisfying $\|\psi-\psi_{h_n}\|_{H^1_0} = o(1)$ and using the same arguments in \eqref{lem-eq-2}, we can deduce 
\[
\int_{\Omega} \nabla W_0\cdot \nabla \psi \,dx = \int_{\Omega}  |W_0|^{p-2}W_0\psi \,dx,
\]
which means that $W_0$ is a weak solution of \eqref{Lane-Emden}.
Since $\{U_{h_n}\} \to W_0$ in $H^1_0$ and $C_{h_n} \to C_0$, we see that 
\[
C_0 = \frac{\|W_0\|_{H^1_0(\Omega)}}{\|W_0\|_{L^p(\Omega)}}
\]
so $W_0$ is also a minimizer of the problem \eqref{eq-0-2}.
From Proposition \ref{non-deg}, we then conclude that $W_0$ is either $U_0$ or $-U_0$.                
\end{proof}

\section{$H^1$ error estimates}
\setcounter{step}{0}
In this section, we compute a sharp $H^1$ convergence rate for $U_h$. 
Choose a sequence $\{h_n\} \to 0$ and a sequence of minimizers $\{U_{h_n}\} \subset V_{h_n}$ of \eqref{eq-0-1} with $h = h_n$ such that
$\lambda_{h_n} = 1$ in \eqref{Euler-Lagrange-fem} and $U_{h_n} \to U_0$ in $H^1_0(\Omega)$, where $U_0$ is a unique positive solution of \eqref{Lane-Emden}.
For notational simplicity, we denote $h_n$ by just $h$.
We divide the proof into the several steps. 
The following elementary estimates will be frequently invoked throughout this section.
\begin{lem}\label{estimates}
For $p > 2$, there exists $C > 0$ independent of $a, b$ such that
\[
\left| |b|^{p-2}b-|a|^{p-2}a \right| \leq  C(|b|^{p-2}+|a|^{p-2})|b-a| 
\]
and
\[
\left||b|^{p-2}b-|a|^{p-2}a-(p-1)|a|^{p-2}(b-a) \right| \leq
\left\{
\begin{array}{rcl}
C(|b|^{p-3}+|a|^{p-3})|b-a|^2 & \text{if} & p \geq 3, \\
C|b-a|^{p-1}& \text{if} &  2< p < 3.
\end{array}\right.
\]
\end{lem}

\begin{step}
There exists a constant $C >0$ independent of $h$  such that
\begin{equation}\label{eq-2-1}
\int_{\Omega} |\nabla (U_h-U_0)|^2 - (p-1)U_0^{p-2} (U_h-U_0)^2 \, dx \leq  Ch\|U_h-U_0\|_{H^1_0(\Omega)} 
+ C\|U_h - U_0\|_{H^1_0(\Omega)}^{\min\{3,p\}}
\end{equation}
\end{step}
\begin{proof}
We recall that 
\begin{equation}\label{eq-33}
\left\{\begin{array}{rcl}
\int_{\Omega} \nabla U_0 \nabla \phi\, dx & = & \int_{\Omega} U_0^{p-1}\phi \, dx \quad \forall ~\phi \in H_0^1 (\Omega), \\
\int_{\Omega} \nabla U_h \nabla \phi_h \, dx & = & \int_{\Omega} |U_h|^{p-2}U_h \phi_h\, dx\quad \forall~ \phi_h \in V_h.
\end{array}\right.
\end{equation}
Then for all $\phi_h \in V_h$,
\begin{equation}\label{eq-32}
\int_{\Omega} \nabla (U_h-U_0) \cdot \nabla \phi_h \,dx = \int_{\Omega} (|U_h|^{p-2}U_h-U_0^{p-1}) \phi_h \, dx.
\end{equation}
From Proposition \ref{approx} and Proposition \ref{regularity}, we see that there exists some $\psi_h \in V_h$ such that $\|\psi_h - U_0\|_{H_0^1} \leq Ch$,
where $C$ depends only on $\Omega$ and $U_0$. Since $U_h \to U$ in $H^1_0(\Omega)$, we may assume $\|U_h-U_0\|_{H^1_0(\Omega)} \leq 1$.
Choosing $\phi_h = U_h-\psi_h$ and using \eqref{eq-32}, we get that
\begin{equation}\label{eq-2-2}
\begin{aligned}
&\int_{\Omega} \nabla (U_h-U_0) \cdot \nabla (U_h-U_0) \, dx - \int_{\Omega} (|U_h|^{p-2}U_h- U_0^{p-1}) (U_h-U_0) \,dx \\
& = \int_{\Omega}\nabla(U_h-U_0)\cdot\nabla(\psi_n-U_0)\,dx- \int_{\Omega} (|U_h|^{p-2}U_h- U_0^{p-1}) (\psi_h-U_0) \,dx. \\
\end{aligned}
\end{equation}
Using Lemma \ref{estimates}, H\"older inequality and Sobolev embedding, we see that
\begin{equation}\label{ineq-step1-0}
\begin{aligned}
&\left|\int_{\Omega}\nabla(U_h-U_0)\cdot\nabla(\psi_n-U_0)\,dx- \int_{\Omega} (|U_h|^{p-2}U_h- U_0^{p-1}) (\psi_h-U_0) \,dx\right| \\
&\leq \|\nabla (U_h-U_0)\|_{L^2}\|\nabla(\psi_h-U_0)\|_{L^2}  + C\int_{\Omega} (|U_h|^{p-2}+U_0^{p-2} )|U_h - U_0||\psi_h-U_0|\,dx \\
&\leq Ch\|U_h-U_0\|_{H^1_0} +C(\|U_h\|_{L^p}^{p-2}+\|U_0\|_{L^p}^{p-2})\|U_h-U_0\|_{L^p}\|\psi_h-U_0\|_{L^p} \\
&\leq Ch\|U_h-U_0\|_{H^1_0}.
\end{aligned}
\end{equation}


Define
\[
I := \int_{\Omega} (|{U_h}|^{p-2}U_h -{U_0}^{p-1}) (U_h-U_0) \,dx - \int_{\Omega} (p-1)U_0^{p-2} (U_h - U_0)^2\, dx.
\]
Now we see from Lemma \ref{estimates} that $I$ satisfies that if $p \geq 3$, then
\[
\begin{aligned}
|I| &\leq C\int_{\Omega}(|U_h|^{p-3}+U_0^{p-3})|U_h-U_0|^3\,dx \\
& \leq C(\|U_h\|_{H^1_0}^{p-3}+\|U_0\|_{H^1_0}^{p-3})\|U_h-U_0\|_{H^1}^3 \leq C\|U_h-U_0\|_{H^1_0}^3,
\end{aligned}
\]
and if $2<p<3$, then 
\[
|I| \leq C\int_{\Omega}|U_h-U_0|^{p}\,dx \leq C\|U_h-U_0\|_{H^1_0}^p.
\]
Inserting this and \eqref{ineq-step1-0} into \eqref{eq-2-2} we find 
\[
\int_{\Omega} |\nabla (U_h-U_0)|^2 - (p-1)U_0^{p-2}(U_h-U_0)^2 \, dx \leq  C h \|U_h-U_0\|_{H^1} +C\|U_h -U_0\|_{H^1}^{\min\{3,p\}},
\]
which shows the proof.
\end{proof}

\begin{step}
There exists a constant $C >0$ independent of $h$ such that
\begin{equation*}
\|U_h -U_0\|_{H_0^1 (\Omega)} \leq Ch.
\end{equation*}
\end{step}
\begin{proof}
We decompose the difference $U_h - U_0$ as the sum of the part tangential to $U_0$ and the part orthogonal to $U_0$. In other words, we choose a constant $\lambda_h \in \R$
and a function $v_h \in H^1_0(\Omega)$ such that 	
\begin{equation}\label{decomposition}
U_h -U_0 = v_h + \lambda_h U_0\quad \text{and} \quad \langle v_h, U_0\rangle_{H^1_0} = 0.
\end{equation}
Observe that
\begin{equation}\label{orthogonality}
0 = \langle v_h, U_0\rangle_{H^1_0} = \int_{\Omega}\nabla v_h \cdot \nabla U_0\,dx = \int_{\Omega} v_h  (-\Delta U_0)\,dx = \int_{\Omega} v_h  U_0^{p-1}\,dx.
\end{equation}
Since $\|v_h\|_{H^1}^2 + \lambda_h^2 \|U_0\|_{H^1}^2 = \|U_h-U_0\|_{H^1}^2 \to 0$, we see that $\|v_h\|_{H^1}, \lambda_h \to 0$.  
In particular we may assume $\|v_h\|_{H^1} < 1, |\lambda_h| < 1$.

We insert \eqref{decomposition} in the left hand side of \eqref{eq-2-1} and use \eqref{orthogonality} to get 
\begin{equation}\label{ineq0}
\begin{split}
& \int_{\Omega} |\nabla (v_h + \lambda_h U_0)|^2 - (p-1) U_0^{p-2} (v_h+\lambda_h U_0)^2 dx 
\\
& = \int_{\Omega} |\nabla v_h|^2 - (p-1)U_0^{p-2} v_h^2\,dx + \lambda_h^2 \int_{\Omega} |\nabla U_0|^2 - (p-1) U_0^{p} \, dx
\\
& = \int_{\Omega} |\nabla v_h|^2 - (p-1)U_0^{p-2} v_h^2\,dx- (p-2) \lambda_h^2 \int_{\Omega} U_0^p dx.
\end{split}
\end{equation}
Then combining \eqref{eq-2-1}, \eqref{ineq0} and Proposition \ref{non-deg}, we get
\[
\int_{\Omega} |\nabla v_h|^2 dx \leq C\int_{\Omega} |\nabla v_h|^2 - (p-1)U_0^{p-2} v_h^2\,dx
\leq C\lambda_h^2 + Ch\|U_h-U_0\|_{H^1_0}+ C\|U_h-U_0\|_{H^1_0}^{\min\{3,p\}}.
\]
Thus, using Young's inequality, we have
\[
\begin{split}
\|v_h\|_{H^1_0}^2 &\leq C \lambda_h^2 +  C h\left( \|v_h\|_{H^1_0} + \lambda_h\right)+C \left( \|v_h\|_{H^1_0}^{\min\{3,p\}}+ \lambda_h^{\min\{3,p\}} \right)
\\
&\leq C \lambda_h^2 +\frac12\left( \|v_h\|_{H^1_0}^2+\lambda_h^2 \right)+ C \left( \|v_h\|_{H^1_0}^{\min\{3,p\}} + \lambda_h^{\min\{3,p\}} \right)  + C h^2,
\end{split}
\]
which can be simplified as
\begin{equation}\label{eq-2-3}
\|v_h\|_{H^1_0}^2 \leq C \left( \lambda_h^2 + \|v_h\|_{H^1_0}^{\min\{3,p\}}+ h^2\right)  
\end{equation}

On the other hand, the second equality of \eqref{eq-33} is written as, for all $\phi_h \in V_h$,
\begin{equation}\label{ineq1}
\int_{\Omega} \nabla ((1+\lambda_h)U_0 + v_h) \cdot \nabla \phi_h \, dx = \int_{\Omega} |(1+\lambda)U_0 + v_h|^{p-2}((1+\lambda_h)U_0 + v_h)\phi_h\, dx
\end{equation}
We again take $\phi_h \in V_h$ such that $\|U_0-\phi_h\|_{H_0^1 (\Omega)} \leq Ch$. Then  arguing similarly as in Step 1, one has 
\begin{equation}\label{ineq2}
\begin{split}
&\int_{\Omega} \nabla ((1+\lambda_h)U_0 + v_h) \cdot \nabla \phi_h \, dx \\
&= (1+\lambda_h)\int_{\Omega} \nabla U_0 \cdot \nabla U_0 \, dx +\int_{\Omega} \nabla ((1+\lambda_h)U_0 + v_h) \cdot \nabla (\phi_h-U_0) \, dx \\
&= (1+\lambda_h)\int_{\Omega} |\nabla U_0|^2\, dx + O(h)
\end{split}
\end{equation}
and
\begin{equation}\label{ineq3}
\begin{split}
&\int_{\Omega} |(1+\lambda_h)U_0 + v_h|^{p-2}((1+\lambda_h)U_0 + v_h)\phi_h\, dx \\
&=\int_{\Omega} |(1+\lambda_h)U_0 + v_h|^{p-2}((1+\lambda_h)U_0 + v_h)U_0\, dx + O(h) \\
&= (1+\lambda_h)^{p-1}\int_{\Omega}U_0^p\,dx +(p-1)(1+\lambda_h)^{p-2}\int_{\Omega}U_0^{p-1}v_h\,dx +II +O(h) \\
&= (1+\lambda_h)^{p-1}\int_{\Omega}U_0^p\,dx +II +O(h),
\end{split}
\end{equation}
where we defined
\[
\begin{aligned}
II:= &\int_{\Omega} |(1+\lambda_h)U_0 + v_h|^{p-2}((1+\lambda_h)U_0 + v_h)U_0\, dx \\
&\qquad-(1+\lambda_h)^{p-1}\int_{\Omega}U_0^p\,dx -(p-1)(1+\lambda_h)^{p-2}\int_{\Omega}U_0^{p-1}v_h\,dx \\
=&\int_{\Omega} |(1+\lambda_h)U_0 + v_h|^{p-2}((1+\lambda_h)U_0 + v_h)U_0\, dx-(1+\lambda_h)^{p-1}\int_{\Omega}U_0^p\,dx.
\end{aligned}
\]
Then using Lemma \ref{estimates} again, we see that
\begin{equation}\label{ineq4}
|II| \leq C\int_{\Omega}(U_0^{p-3}+|v_h|^{p-3})U_0v_h^2\,dx 
\leq C(\|U_0\|_{L^p}^{p-3}+\|v_h\|_{L^p}^{p-3})\|U_0\|_{L^p}^p\|v_h\|_{L^p}^2 \leq C\|v_h\|_{H^1_0}^2
\end{equation}
if $p \geq 3$ and
\begin{equation}\label{ineq5}
|II| \leq C\int_{\Omega}U_0v_h^{p-1}\,dx
\leq C\|U_0\|_{L^p}^p\|v_h\|_{L^p}^{p-1} \leq C\|v_h\|_{H^1_0}^{p-1}
\end{equation}
Combining \eqref{ineq1}--\eqref{ineq5}, we have
\[
\begin{aligned}
\left|((1+\lambda_h)^{p-1}-(1+\lambda_h))\right|\int_{\Omega}|\nabla U_0|^2\,dx  &\leq  \left|(1+\lambda_h)^{p-1}\int_{\Omega}U_0^p\,dx-(1+\lambda_h)\int_{\Omega}|\nabla U_0|^2\,dx\right| \\
&\leq C(h +\|v_h\|_{H^1_0}^{\min\{p-1,2\}}),
\end{aligned}
\]
which simplifies to
\[
(1+\lambda_h)^{p-2}-1 \leq C(h+\|v_h\|_{H^1_0}^{\min\{p-1,2\}}).
\]
Invoking mean value theorem, there exists some $\xi_h$ between $0$ and $\lambda_h$ such that
\[
(1+\lambda_h)^{p-2}-1 = (p-2)(1+\xi_h)^{p-3}\lambda_h,
\]
from which we see that
\[
|\lambda_h| \leq \frac{C}{(p-2)|1+\xi_h|^{p-3}}(h+\|v_h\|_{H^1_0}^{\min\{p-1,2\}}) \leq C(h+\|v_h\|_{H^1_0}^{\min\{p-1,2\}}),
\]
because $\xi_h \to 0$. 
Combining this with \eqref{eq-2-3}, we arrive at the following estimate
\begin{equation*}
\begin{aligned}
\|v_h\|_{H^1}^2  &\leq C \left( \lambda_h^2 + \|v_h\|_{H^1_0}^{\{3,p\}} + h^2\right) \\
&\leq C\left( h^2 + \|v_h\|_{H^1_0}^{\min\{2(p-1),4\}}+ \|v_h\|_{H^1_0}^{\min\{3,p \}} + h^2\right),
\end{aligned}
\end{equation*}
Since $\|v_h\|_{H^1_0(\Omega)} \to 0$ and $p > 2$, this shows
\[
\|v_h\|_{H^1_0(\Omega)}^2 \leq Ch^2 \quad \text{and} \quad \lambda_h^2 \leq Ch^2.
\]
Thus we finally conclude that
\[
\|U_h-U_0\|_{H_{\Omega}^1}^2 = \|v_h\|_{H_{\Omega}^1}^2+\lambda_h^2 \leq Ch^2.
\]
This completes the proof.
\end{proof}

\section{$L^2$ error estimates}
In this section, we prove the $L^2$ error estimate  for $U_h$. 
Choose a sequence $\{h_n\} \to 0$ and a sequence of minimizers $\{U_{h_n}\} \subset V_{h_n}$ of \eqref{eq-0-1} with $h = h_n$ such that
$\lambda_{h_n} = 1$ in \eqref{Euler-Lagrange-fem} and $U_{h_n} \to U_0$ in $H^1_0(\Omega)$, where $U_0$ is a unique positive solution of \eqref{Lane-Emden}.
As in the previous section, we shall denote $h_n$ by just $h$.
Consider the linear operator $\mathcal{L} \colon H^2(\Omega) \to L^2(\Omega)$ defined by
\[
\mathcal{L} := -\Delta -(p-1)U_0^{p-2}, 
\]
which is the linearized operator of the equation \eqref{Lane-Emden} at $U_0$. We prepare a lemma. 
\begin{lem}
For given data $f \in L^2(\Omega)$, there exists a unique solution $w \in H^1_0(\Omega) \cap H^2(\Omega)$ of the problem 
\begin{equation}\label{linear-problem}
\mathcal{L}[w] = f \ \text{ in } \Omega, \quad w  = 0 \ \text{ on } \partial\Omega.
\end{equation}
such that the following estimate holds for some $C > 0$ independent of $f$: 
\begin{equation}\label{elliptic-est}
\|w\|_{H^2(\Omega)} \leq C\|f\|_{L^2(\Omega)}.
\end{equation}

\end{lem}
\begin{proof}
By Proposition \ref{non-deg}, the operator $\mathcal{L}$ has no kernel element so by the Fredholm alternative theory,  
there exists a unique solution  $w \in H^1_0 \cap H^2$ of the problem \eqref{linear-problem}.
We multiply the equation \eqref{linear-problem} by $U_0$ and integrate by parts to see
\[
\begin{aligned}
\int_{\Omega}\nabla w \cdot \nabla U_0\,dx & = \int_{\Omega}(p-1)wU_0^{p-1}\,dx +\int_{\Omega}fU_0\,dx \\
&= \int_{\Omega}(p-1)w(-\Delta U_0)\,dx +\int_{\Omega}fU_0\,dx \\
&= (p-1)\int_{\Omega}\nabla w \cdot \nabla U_0\,dx +\int_{\Omega}fU_0\,dx
\end{aligned}
\]
so we have
\begin{equation}\label{linear-problem-eq0}
\langle w, U_0\rangle_{H^1_0} = \int_{\Omega}\nabla w \cdot\nabla U_0\,dx = \frac{1}{2-p}\int_{\Omega}fU_0\,dx
\end{equation}
Now we consider the orthogonal decomposition of $w$ by $w = v +\lambda U_0$ such that $\langle v, U_0 \rangle_{H^1_0} = 0$ and, consequently $\langle v, \mathcal {L}[U_0]\rangle_{L^2} = 0$ holds. 
Then one has from \eqref{linear-problem-eq0} that
\begin{equation}\label{linear-problem-eq1}
|\lambda| =  \left|\langle w, U_0\rangle_{H^1_0}/\|U_0\|_{H^1_0}^2\right| \leq C\|f\|_{L^2}.
\end{equation}
On the other hand, after multiplying \eqref{linear-problem} by $w$ we use the decomposition of $w$ and Proposition \ref{non-deg} to get
\[
\begin{aligned}
\int_{\Omega}f(v+\lambda U_0)\,dx & = \int_{\Omega}\mathcal{L}[v+\lambda U_0](v+\lambda U_0)\,dx \\
& = \int_{\Omega}\mathcal{L}[v]v\,dx +2\lambda\int_{\Omega}\mathcal{L}[U_0]v\,dx  +\lambda^2\int_{\Omega}\mathcal{L}[U_0]U_0\,dx  \\
& \geq C\|v\|_{H^1_0}^2 +(2-p)\lambda^2\int_{\Omega}U_0^p\,dx.
\end{aligned}
\]
Combining this with \eqref{linear-problem-eq1}, we have from the Young's inequality that
\[
\begin{aligned}
\|v\|_{H^1_0}^2 &\leq (\|v\|_{L^2}+C|\lambda|)\|f\|_{L^2}+C|\lambda|^2 \\
&\leq \frac12\|v\|_{L^2}^2 +C\|f\|_{L^2}^2,
\end{aligned}
\]
which shows that $\|v\|_{H^1_0} \leq C\|f\|_{L^2}$ by the Sobolev embedding. 
Since $\|w\|_{H^1_0}^2 = \|v\|_{H^1_0}^2+\lambda^2\|U_0\|_{H^1_0}$, we also get $\|w\|_{H^1_0} \leq C\|f\|_{L^2}$.
Considering the equation 
\[
-\Delta w = (p-1)U_0^{p-2}w+f
\]
and invoking Proposition \ref{regularity}, we finally have
\[
\|w\|_{H^2} \leq C(\|(p-1)U_0^{p-2}w\|_{L^2}+\|f\|_{L^2}) \leq C\|f\|_{L^2}.
\]
This completes the proof. 
\end{proof}
Now we begin the proof of the $L^2$ error estimate of \eqref{eq-16}. Let $w_h \in H^2$ be a unique solution of the problem
\[
\mathcal{L}[w] = U_h-U_0 \ \text{ in } \Omega, \quad w  = 0 \ \text{ on } \partial\Omega
\]
such that the estimate $\|w_h\|_{H^2} \leq C\|U_h-U_0\|_{L^2}$ holds true.
Then one has 
\begin{equation}\label{eq-45}
\begin{split}
\int_{\Omega} (U_h-U_0)^2 dx & = \int_{\Omega}\mathcal{L}[w_h](U_h-U_0)\,dx \\
&= \int_{\Omega}\nabla w_h \cdot \nabla(U_h-U_0)\, dx- (p-1)\int_{\Omega}w_hU_0^{p-2}(U_h-U_0) dx
\end{split}
\end{equation}
Take $\phi_h \in V_h$ satisfying $\|\phi_h-w_h\|_{H^1_0} \leq Ch\|w_h\|_{H^2}$.
Then one must have
\begin{equation*}
\int_{\Omega} \nabla (U_h-U_0) \nabla \phi_h dx = \int_{\Omega} (|U_h|^{p-2}U_h -U_0^{p-1}) \phi_h dx.
\end{equation*}
Combining this with \eqref{eq-45}, and then using Lemma \ref{estimates} and $H^1$ convergence rate of $U_h$ obtained in the previous section, we obtain
\begin{equation*}
\begin{split}
\int_{\Omega} (U_h-U_0)^2\, dx & = \int_{\Omega} \nabla (w_h-\phi_h) \cdot \nabla (U_h-U_0) dx -(p-1)\int_{\Omega}(w_h-\phi_h)U_0^{p-2}(U_h-U_0)\,dx
\\
&\qquad +\int_{\Omega}\left(|U_h|^{p-2}U_h-U_0^{p-1}-(p-1)U_0^{p-2}(U_h-U_0)\right)\phi_h\,dx \\
&\leq \|w_h-\phi_h\|_{H^1_0}\|U_h-U_0\|_{H^1_0} +\|w_h-\phi_h\|_{L^p}\|U_0\|_{L^p}^{p-2}\|U_h-U_0\|_{L^p} \\
&\qquad +\left\{\begin{array}{lcl}
C(\|U_h\|^{p-3}_{L^p}+\|U_0\|^{p-3}_{L^p})\|U_h-U_0\|_{L^p}^2\|\phi_h\|_{L^p} & \text{if} & p \geq 3, \\
C\|U_h-U_0\|_{L^2}^{p-1}\|\phi_h\|_{L^{\frac{2}{3-p}}} & \text{if} &  2 < p < 3, 
\end{array}\right. \\
&\leq \left\{\begin{array}{lcl}
Ch\|w_h-\phi_h\|_{H^1_0} +Ch^2\|\phi_h\|_{H^1_0} & \text{if} & p \geq 3, \\
Ch\|w_h-\phi_h\|_{H^1_0} +C\|U_h-U_0\|_{L^2}^{p-1}\|\phi_h\|_{H^1_0} & \text{if} & 2 < p < 3.
\end{array}\right.
\end{split}
\end{equation*}
From the fact that $\|\phi_h-w_h\|_{H^1_0} \leq Ch\|w_h\|_{H^2}$, we see that $\|\phi_h\|_{H^1_0} \leq C\|w_h\|_{H^2}$, and consequently,
using estimate $\|w_h\|_{H^2} \leq C\|U_h -U_0\|_{L^2}$ from \eqref{elliptic-est}, one has
\[
\int_{\Omega} (U_h-U_0)^2\, dx \leq \left\{\begin{array}{lcl}
Ch^2\|U_h-U_0\|_{L^2}  & \text{if} & p \geq 3, \\
Ch^2\|U_h-U_0\|_{L^2} +C\|U_h-U_0\|_{L^2}^p & \text{if} & 2 < p < 3.
\end{array}\right.
\]
Then we see that in any case the desired $L^2$ convergence rate is obtained.

\section{The uniform $L^{\infty}$ estimate}

This section is devoted to prove the uniform $L^{\infty}$ estimate of $U_h$.
We recall that
\[
\int_{\Omega} (\nabla U_h \cdot \nabla \phi_h) \, dx = \int_{\Omega} |U_h|^{p-2}U_h \phi_h \, dx \quad \forall ~ \phi_h \in V_h.
\]
We define $v_h \in H_0^1 (\Omega)$ as the unique solution of 
\[
-\Delta v = |U_h|^{p-2}U_h \ \text{ in } \Omega, \quad v \in H^1_0(\Omega).
\]
In particular, $v_h$ satisfies 
\[
\int_{\Omega} \nabla v_h \cdot \nabla \phi \, dx = \int_{\Omega} |U_h|^{p-2}U_h \phi \, dx,\quad \forall~ \phi \in H_0^1 (\Omega).
\]
Then one must have
\begin{equation*}
\int_{\Omega} \nabla (U_h - v_h) \nabla \phi dx =0\quad \forall~ \phi \in V_h,
\end{equation*}
which means that $U_h$ is the $H^1$ projection of $v_h$ to the finite element space $V_h$.
Thus we have from Proposition \ref{approx} that
\begin{equation}\label{scott}
\|U_h\|_{W^{1,q}(\Omega)} \leq C_q\|v_h\|_{W^{1,q}(\Omega)}
\end{equation}
as long as the right hand side is finite.
Let $G(x,y)$ denote the Green function of $-\Delta$ on $\Omega$ with the Dirichlet boundary condition. 
Then $v_h$ is given by
\[
v_h (x) = \int_{\Omega} G(x,y) |U_h|^{p-2}U_h (y) dy.
\]
Since we have the following uniform gradient estimate of Green function \cite{Fromm, GW}: 
\[
\left|\nabla_x G(x,y) \right|\leq C\frac{1}{|x-y|},
\]
the Hardy-Littlewood-Sobolev inequality implies that
\[
\|v_h \|_{W^{1,q}(\Omega)} \leq C \| |U_h|^{p-1}\|_{L^r (\Omega)}
\]
for any $q>r>1$ satisfying $\frac{1}{r} - \frac{1}{q} =\frac{1}{2}$. Let us choose $r= 3/2$ and $q = 6$. Then, 
\[
\| v_h \|_{W^{1,6}} \leq C \| |U_h|^{p-1}\|_{L^{3/2}} = \|U_h\|_{L^{3(p-1)/2}}^{p-1}\leq \|U_h\|_{H^1_0}^{p-1}.
\]
We combine this with \eqref{scott} and use the Sobolev embedding to conclude that
\[
\|U_h\|_{L^\infty} \leq C\|U_h\|_{H^1_0}^{p-1}. 
\]
This completes the proof. 

\section{Numerical results}

In the numerical implementation, we computed the approximate solutions in the case $p =4$, $n=2$ and $\Omega = (0,1)^2$. 
Since we do not have an explicit formula for the original solution, we computed the error $\|u_h -u_{h/2}\|_{L^2 (\Omega)}$ and $\|u_h - u_{h/2}\|_{H^1}$, where $h$ is the length of the triangle. We conducted the numeric with $h$ given by $h_j = 2^{-j}$ for $1 \leq j \leq 7$. Since we do not have an explicit form of the exact solution,  we computed the decrease of the error. Namely, for each $2 \leq j \leq 7$, we calculated $R_j^0$ and $R_j^1$ given as
\begin{equation*}
R_j^0 =\log_2 \left( \frac{\|u_{h_j}-u_{h_{j+1}}\|_{L^2}}{\|u_{h_{j-1}}-u_{h_j}\|_{L^2}}\right)\quad \textrm{and}\quad R_j^1 = \log_2 \left( \frac{\|u_{h_j}-u_{h_{j+1}}\|_{H^1}}{\|u_{h_{j-1}}-u_{h_j}\|_{H^1}}\right).
\end{equation*}
To obtained the numerical solution for the nonlinear problem, we iterated combination of the gradient descent method and the $L^{p+1}(\Omega)$  norm normalization: First fix an initial data $u_0 \in L^{p+1}(\Omega)$, and then we iterate the following two steps:
\begin{enumerate}
\item We choose a small value $\eta>0$. Then we consider the gradient descent of the energy function $E(u)$, i.e., 
\begin{equation}
\nabla E(u) = u - (-\Delta)^{-1} (|u|^{p-1}u)
\end{equation}
and substitute $u \rightarrow u - \delta \nabla E(u)$. 
\item Next we normalize the $L^{p+1}(\Omega)$-norm as
\begin{equation}
u \rightarrow \frac{u}{\|u\|_{L^{p+1}(\Omega)}}.
\end{equation}
\end{enumerate}
In the above, to obtain the function $w=(-\Delta)^{-1}(|u|^{p-1} u)$, we computed the approximate function $w_h \in V_h$ such that
\begin{equation}
\int_{\Omega} \nabla w_h \nabla \phi dx = \int_{\Omega} \phi |u|^{p-1} u(x) dx\quad \forall ~\phi \in V_h.
\end{equation}
We chose the domain $\Omega = [0,1]\times [0,1]$ and took the initial data $u_0 \in V_h$ so that $u_0 =1$ on all the interior nodes, and $u_0 = 0$ on the boundary nodes. We then nomalized the $L^{p+1}(\Omega)$ of $u_0$. For the iteration, we took $\eta =0.2$ and iterated the above two steps for $60$ times. We examined two cases $p=3$ and $p=10$.
\begin{figure}[htbp]\includegraphics[height=3.2cm, width=3.2cm]{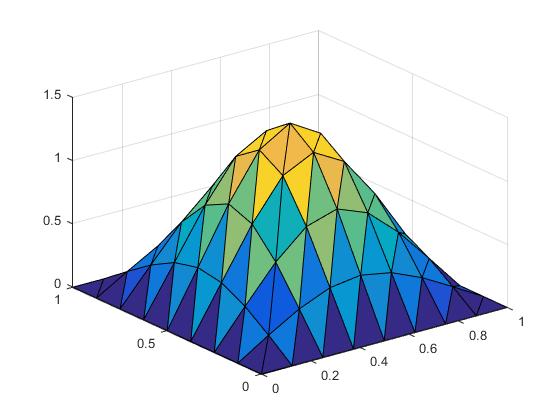}\includegraphics[height=3.2cm, width=3.2cm]{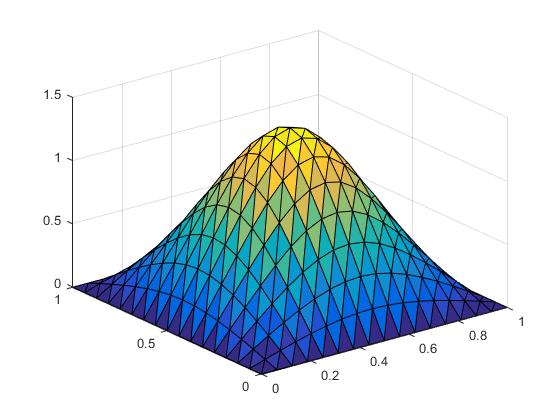}\includegraphics[height=3.2cm, width=3.2cm]{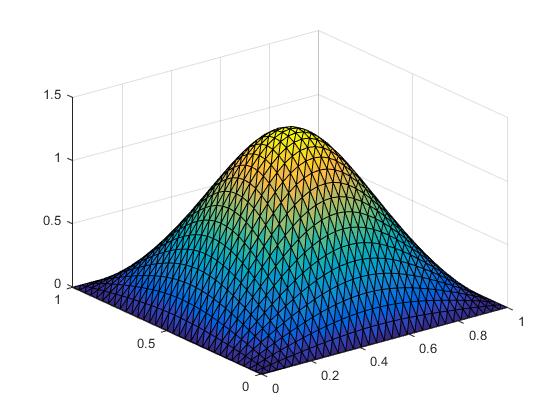}\includegraphics[height=3.2cm, width=3.2cm]{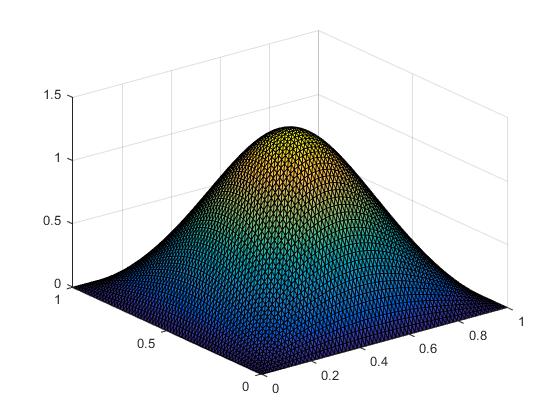}
\vspace{-0.3cm}
\caption{The approximate solutions for $p=3$.}\label{pic}
\end{figure}
\

Figure \ref{pic} shows the solutions with $p=3$ computed with mesh sizes $1/2^3$, $1/2^4$, $1/2^5$, and $1/2^6$.
\begin{table}[htbp]
\begin{center}
\caption{The $L^2$ and $H^1$ errors for the case $p=3$.}\label{tab1}
\vspace{0.2cm}
\begin{tabular}{|c||c c|c c|}
\hline
\vspace{-0.4cm}
 & & & & \\
$h_j$ & \multicolumn{1}{c}{$\|u_{h}-u_{h/2}\|_{L^2}$} & \multicolumn{1}{c|}{$Rate (R_j^0)$}
& \multicolumn{1}{c}{$\|u_{h}-u_{h/2}\|_{H^1}$} & \multicolumn{1}{c||}{$Rate (R_j^1)$}
\\
\hline\hline
\vspace{-0.3cm}
 & & & &  \\
$2^{-1}$ & 4.5500E-01	& - &	2.5190E+00&	-\\
$2^{-2}$ & 7.9379E-02	&2.51&	1.0314E+00&	1.40\\
$2^{-3}$ & 1.9137E-02	&2.05&4.8709E-01&	1.08\\
$2^{-4}$ & 4.9273E-03	&1.95&	2.4000E-01&	1.02\\
$2^{-5}$ & 1.2837E-03 &1.94&	1.1954E-01&	1.01\\
$2^{-6}$ & 3.4450E-04	&1.89&5.9711E-02&	1.00\\
$2^{-7}$ & 9.9473E-05     &1.79&2.9847E-02&	1.00\\
\hline
\end{tabular}
\end{center}
\end{table}

\begin{table}[htbp]
\begin{center}
\caption{The $L^2$ and $H^1$ errors for the case $p=10$.}\label{tab2}
\vspace{0.2cm}
\begin{tabular}{|c||c c|c c|}
\hline
\vspace{-0.4cm}
 & & & & \\
$h_j$ & \multicolumn{1}{c}{$\|u_{h}-u_{h/2}\|_{L^2}$} & \multicolumn{1}{c|}{$Rate (R_j^0)$}
& \multicolumn{1}{c}{$\|u_{h}-u_{h/2}\|_{H^1}$} & \multicolumn{1}{c||}{$Rate (R_j^1)$}
\\
\hline\hline
\vspace{-0.3cm}
 & & & &  \\
$2^{-1}$ & 6.3268E-01	& - &	3.3675E+00&	-\\
$2^{-2}$ & 1.4409E-01	&2.13&	1.0837E+00&	1.60\\
$2^{-3}$ & 4.9285E-02	&1.55&5.7721E-01&	0.91\\
$2^{-4}$ & 1.6337E-02	&1.59&	3.0117E-01&	0.94\\
$2^{-5}$ & 4.7800E-03 &1.77&	1.5087E-01&	1.00\\
$2^{-6}$ & 1.2789E-03	&1.90&7.5277E-02&	1.00\\
$2^{-7}$ & 3.3675E-04     &1.92&3.7605E-02&	1.00\\
\hline
\end{tabular}
\end{center}
\end{table}

 Table \ref{tab1} shows the error of $L^2$ and $H^1$ with the ratios for the case $p=3$, and Table \ref{tab2} shows the correponding errors and ratios for the case $p=10$.

\newpage
\appendix
\section{Analytic tools}

In the appendix, we arrange some auxiliary tools which are required to handle some analytic issues arising when we prove our main results.

\begin{prop}[\cite{Bartels}, \cite{Scott}]\label{approx}
For any $u \in H_0^1(\Omega)$, define $P_h(u)$ by the projection of $u$ to $V_h$ in $H^1_0(\Omega)$.  In other words, $P_h(u)$ is a unique element in $V_h$ satisfying
\[
\int_{\Omega}u\phi_h\,dx = \int_{\Omega} P_h(u)\phi_h\,dx \ \text{ for all } \phi_h \in V_h.
\] 
Then the following estimates hold:
\[
\|u-P_h(u)\|_{H_0^1(\Omega)} = o(1), \quad \text{and} \quad \|u-P_h(u)\|_{L^2(\Omega)} = O (h) \|u\|_{H^1_0(\Omega)}
\quad \text{as } h \to 0.
\]
If $u \in H_0^1(\Omega) \cap H^2(\Omega)$ the following estimates hold:
\[
\|u-P_h(u)\|_{H^1(\Omega)} = O(h) \|u\|_{H^2(\Omega)} \quad \text{and} \quad \|u-P_h(u)\|_{L^2(\Omega)} = O(h^2) \|u\|_{H^2(\Omega)}
\quad \text{as } h \to 0.
\]
If $u \in W^{1,q}_0(\Omega)$ for some $q \geq 2$, the following estimate holds (scott):
\[
\|P_h(u)\|_{W^{1,q}(\Omega)} \leq C\|u\|_{W^{1,q}(\Omega)}
\]
for some $C > 0$ independent of $h$.
\end{prop}


\begin{prop}[\cite{Grisvard}]\label{regularity}
Let $\Omega \subset \R^2$ be a bounded convex domain with a polygonal boundary.
For given $f \in L^2(\Omega)$, let $u \in H^1_0(\Omega)$ be a weak solution of the problem
\[
-\Delta u = f \ \text{ in } \Omega, \quad u \in H^1_0(\Omega)
\]
Then $u$ belongs to $H^2(\Omega)$, and there exists a constant $C>0$ such that 
\begin{equation*}
\|u\|_{H^2 (\Omega)} \leq C\|f\|_{L^2 (\Omega)}.
\end{equation*}
\end{prop}

\begin{prop}[\cite{CHS}, \cite{Lin}]\label{non-deg}
Let $\Omega \subset \R^2$ be a bounded convex domain and  $p \in (2, \infty)$.
Let $U$ be a minimizer of the problem
\[
C(\Omega, p) = \inf\left\{\frac{\|\nabla u\|_{L^2(\Omega)}}{\|u\|_{L^p(\Omega)}} ~\Big|~ u \in H^1_0(\Omega), u \neq 0\right\}
\]
satisfying 
\begin{equation}\label{non-deg-Lane-Emden}
-\Delta u = |u|^{p-2}u \ \text{ in } \Omega.
\end{equation}
Then there holds the following:
\begin{enumerate}[{\em(i)}]
\item $U$ is sign definite and unique up to a sign. 
\item$U$ is non-degenerate. In other words, the linearized equation of \eqref{non-deg-Lane-Emden} at $U$, i.e.,
\[
\Delta \phi +(p-1)U^{p-2}\phi = 0 \ \text{ in } \Omega, \quad \phi \in H^1_0(\Omega) 
\]
admits only the trivial solution.
\item The following inequality
\begin{equation}\label{non-deg-ineq}
\int_{\Omega} |\nabla \phi|^2 - (p-1) U^{p-2} \phi^2 dx \geq C \int_{\Omega} |\nabla \phi|^2 dx 
\end{equation}
holds true for any $\phi \in H^1_0(\Omega)$ satisfying $\langle \phi, U\rangle_{H^1_0(\Omega)} = 0$ and some $C > 0$ independent of $\phi$.
\end{enumerate}
\end{prop}
\begin{rem}
The statements (i) and (ii) is proved in \cite{Lin}. The statement (iii) is a natural consequences of (ii). We refer to \cite{CHS} for the rigorous arguments of the proof.   
\end{rem}

\end{document}